\documentclass{amsproc}

\usepackage{amssymb,amsmath,amsthm,amscd}
\newtheorem{theorem}{Theorem}
\newtheorem{lemma}{Lemma}
\newtheorem{proposition}{Proposition}
\newtheorem{definition}{Definition}
\newtheorem{corollary}{Corollary}
\newtheorem*{notation}{Notation}

\newtheorem{example}{Example}

\begin{document}

\title[Differential Modules]{Differential Projective Modules  over Differential Rings, II}

\author{Lourdes Juan}
\address{Department of Mathematics\\
        Texas Tech University\\
        Lubbock TX\\}
\curraddr{}
\email{lourdes.juan@ttu.edu}
\thanks{}

\author{Andy Magid}
\address{Department of Mathematics\\
        University of Oklahoma\\
        Norman OK 73019\\
        }
\curraddr{}
\email{amagid@ou.edu}
\thanks{}
 
\date{October 5, 2020}

\begin{abstract} 
Differential modules over a commutative differential ring $R$ which are finitely generated projective as ring modules, with differential homomorphisms, form an additive category, so their isomorphism classes form a monoid. We study the quotient monoid   by the submonoid of isomorphism classes of free modules with component wise derivation. This quotient monoid has the reduced $K_0$ of $R$ (ignoring the derivation) as an image and contains the reduced $K_0$ of the constants of $R$ as its subgroup of units. This monoid provides a description of the isomorphism classes of differential projective $R$ modules up to an equivalence.
\end{abstract}

\maketitle

\section{Introduction} 

Let $R$ be a commutative differential ring with derivation $D_R$, or just $D$ if the context is clear. A \emph{differential projective module} over $R$ is a finitely generated projective $R$ module $M$ and an additive endomorphism $D_M$, or just $D$ if the context is clear, such that $D_M(rm)=D_R(r)m+rD_M(m)$ for all $r \in R$ and $m \in M$. For example, if $A \in M_n(R)$ defining $D$ on the free $R$ module $R^n$ of column $n$-tuples of elements of $R$ by $D((x_1, \dots, x_n)^T)=(x_1^\prime, \dots, x_n^\prime)^T + A(x_1, \dots, x_n)^T$ makes the free module into a differential projective module denoted  $(R^n, A)$ whose differential structure is denoted $D_A$. (Note: in \cite{jm2} we used a similar construction for differential structures on free modules $R^{(n)}$ of row $n$-tuples.) When $A$ is a zero matrix, denoted $0$ regardless of size, we call the differential module, and those differentially isomorphic to it, \emph{differentially trivial}.  Thus $(R,0)^{(n)}$, which is differentially isomorphic to $(R^n,0)$ is differentially trivial.

In \cite{jm2} we investigated differential projective modules and looked at their classification. It is our goal in this work to advance this classification. In \cite{jm2}, we took a $K$ theoretic approach. We review briefly that approach and its difficulties.

The isomorphism classes of differential projective modules, with an addition operation induced from direct sum, form a commutative monoid with identity the class of the zero module; the most general group to which this monoid maps is denoted $K_0^\text{diff}(R)$ \cite[\S 4]{jm2}. We use brackets to denote the image of a differential projective module in $K_0^\text{diff}(R)$. For differential projective modules $M$ and $N$, $[M]=[N]$ provided there is a differential projective module $P$ such that $M \oplus P$ is differentially isomorphic to $N \oplus P$. There is a differential projective module $Q$ such that $P \oplus Q$ is differentially isomorphic to $(R^n,C)$ for suitable $n$ and $C$  \cite[Corollary 3, p. 4342]{jm2} so we conclude that $[M]=[N]$ provided there is a $C$ with $M \oplus  (R^n,C)$ differentially isomorphic to $N \oplus (R^n,C)$.

Forgetting the differential structure defines a homomorphism from this differential $K$ group to the usual $K$ group, $K_0^\text{diff}(R) \to K_0(R)$. This is an epimorphism by  \cite[Theorem 2, p. 4341]{jm2}, and every element of the kernel has the form $[(R^n,C)]-[(R^n,D)]$ by \cite[\S 4]{jm2}. Informally, these results say that, setting aside the non--differential $K$-theory of $R$, the differential $K$ theory of $R$ is controlled by the classes of differential projective modules whose underlying $R$ module is free. Thus the classification problem, using $K$ theory, comes down to the relation of declaring $(R^n,A)$ and $(R^n,B)$ equivalent provided there is a $C$ with $(R^n,A) \oplus (R^m,C)$ differentially isomorphic to
$(R^n,B) \oplus (R^m,C)$.

Differential projective modules $(R^n,A)$ and $(R^n,B)$ are differentially isomorphic if there is an invertible matrix $T$ with $T^\prime = AT-TB$.  Thus $(R^n,A) \oplus (R^m,C)$ is differentially isomorphic to $(R^n,B) \oplus (R^m,C)$ provided there is an invertible matrix $S$ with $S^\prime=(A \oplus C)S -S(B \oplus C)$. Deciding whether  $(R^n,A)$ and $(R^n,B)$ are equivalent thus requires considering all possible $C$'s and $S$'s for which this equation holds.

To simplify this differential matrix algebra problem, we propose a finer classification, where we restrict the $C$'s to be zero matrices. (Example \ref{E:nocancellation} below shows that this classification can actually be finer.) This yields  a quotient of the monoid of isomorphism classes which is a monoid, but need not be a group.

Starting with the monoid of isomorphism classes of differential projective modules we look at the quotient by the submonoid of isomorphism classes of modules of the form $(R^n,0)$. While this quotient differential monoid is not always a group, it maps surjectively to the similar object defined for $R$ ignoring the derivation, which is a group, and which we show is isomorphic to $K_0(R)/<[R]>$. When $R$ is connected, this quotient is the kernel of the rank function. Moreover, we show that the group of units of the quotient monoid is isomorphic to  $K_0(R^D)/<[R^D]>$. Also, we will see that the part of the quotient monoid not described by the $K$ theory of $R$ and $R^D$ is captured by the quotient of the submonoid  of differential isomophism classes of the modules of the form $(R^n,A)$. (This approach is similar to the one adopted for differential Azumaya algebras in \cite{m}, an approach taken because of the difficulties similar to those noted above that appear in \cite{jm}).

\section{The Monoids}

A monoid is a set endowed with an associative binary operation and an identity element. Every monoid considered here will commutative.  We use addition for the operation, and $0$ for the identity. Except for the notation for the operation, the following observations are identical with \cite{m}: 
Throughout this section we will be concerned with monoids of isomorphism classes of differential modules and various quotients thereof.  If $N \subseteq M$ is a submonoid, by $M/N$ we mean the set of equivalence classes on $M$ under the equivalence relation $m_1 \sim m_2$ if there are $n_1, n_2 \in N$ such that $m_1+n_1=m_2+n_2$. Sums of equivalent elements are equivalent, which defines a commutative operation on $M/N$ by adding representatives of equivalence classes, and the equivalence class of the identity is an identity. Thus $M/N$ is a monoid.  The projection $p: M \to M/N$ ($p(m)$ is the class of $m$)  is an epimorphism, $p(N)=0$, and any monoid homomorphism $q: M \to Q$ with $q(N)=0$ factors uniquely through $p$.
Note, however,  that we can have $M/N=0$ but $N \neq M$, for example $\mathbb Q^*/\mathbb Z^\times$, the multiplicative group of non-zero rationals modulo the multiplicative monoid of non-zero integers.. Thus $p^{-1}(0)$ can be strictly larger than $N$.

If $N$ is a subgroup of the monoid $M$, $M/N$ is the set of $N$ orbits in $M$ where $N$ acts by translation.

The set of invertible elements $U(M)$ of $M$ is its maximal subgroup. The set of elements $a \in M$ whose equivalence classes in $M/N$ are invertible is 
\[
\{ a \in M \vert \exists b \in M, n \in N \text{ such that } a+b+n \in N\}.
\]
The monoid $M/U(M)$ has no invertible elements except the identity.

We will be considering the following monoids. In the definitions we sometimes consider $R$ as a commutative ring and sometimes as a differential ring, and similarly for differential projective $R$ modules. 

\begin{definition} \label{themonoids} 
$P(R)$ denotes the monoid of isomorphism classes of finitely generated projective $R$ modules with the operation induced by direct sum and the identity the isomorphism class of $0$.

$P^\text{diff}(R)$ denotes denotes the monoid of isomorphism classes of differential finitely generated projective $R$ modules with the operation induced by direct sum and the identity the isomorphism class of $0$.

$F(R)$ denotes the submonoid of $P(R)$ whose elements are isomorphism classes of free $R$ modules.

$F^\text{diff}(R)$ denotes the submonoid of $P^\text{diff}(R)$ whose elements are isomorphism classes of differential projective modules of the form $(R^n, Z)$.

$F_0^\text{diff}(R)$ denotes the submonoid of $P^\text{diff}(R)$ whose elements are isomorphism classes of differential projective modules  of the form $(R^n, 0)$.

$PC(R)$  denotes $P(R)/F(R)$  

$PC^\text{diff}(R)$ denotes $P^\text{diff}(R)/F_0^\text{diff}(R)$.

$FC^\text{diff}(R)$ denotes $F^\text{diff}(R)/F_0^\text{diff}(R)$.

$PC(R)$ is called the \emph{projective class monoid} of $R$.

$PC^\text{diff}(R)$ is called the \emph{differential projective class monoid} of $R$.

$FC^\text{diff}(R)$ is called the \emph{ diffferential  free module class monoid} of $R$.
\end{definition}

All of the monoids in Definition \ref{themonoids} are functors.

Using Definiton \ref{themonoids} and the definition of quotient monoids we see that if $P$ and $Q$ are finitely generated projective $R$ modules (respectively differential finitely generated projective $R$ modules) then their classes in $PC(R)$ (respectively $PC^\text{diff}(R)$) will be equal provided for some $m$ and $n$  $P \oplus R^n$ is isomorphic to $Q \oplus R^m$ (respectively $P \oplus (R^n, 0)$ is differentially isomorphic to $Q \oplus (R^m, 0)$). In particular, $P$ is zero in $PC(R)$ (respectively $PC^\text{diff}(R)$) provided for some $m$ and $n$  $P \oplus R^n$ is isomorphic to $ R^m$ (respectively $P \oplus (R^n, 0)$ is differentially isomorphic to $ (R^m, 0)$). In this case $P$ is called stably free (respectively \emph{stably differentially trivial}).

\begin{proposition} \label{reducedK} $PC(R)$ is a group isomorphic to $K_0(R)/<[R]>$. 
\end{proposition}

\begin{proof} Let $P$ be a finitely generated projective $R$ module. The existence of a $Q$ such that $P \oplus Q$ is a free module shows that the class of $Q$ is an inverse to the class of  $P$ in $P(R)/F(R)$. Thus $PC(R)$ is a group. Hence the map $P(R) \to PC(R)$ factors through $K_0(R)$, and since $[R] \mapsto 0$ it factors through $K_0(R)/<[R]>$.  On the other hand, the monoid map $P(R) \to K_0(R)$ maps $F(R)$ into $<[R]>$ so that $P(R) \to K_0(R)/<[R]>$ factors through $P(R)/F(R)$. The composites in either order are the identity, so $PC(R)$ is isomorphic to $K_0(R)/<[R]>$. \end{proof}

Suppose $R$ is connected, so that the rank homomorphism $rk: K_0(R) \to \mathbb Z$ is defined.
By \cite[Proposition 3.2 p. 460]{b2} $K_0(R)/<[R[>$ is isomorphic to $\text{Ker}(rk)$.

The kernel of $rk$ is denoted $Rk_0(R)$ in \cite[p. 459]{b2}. Sometimes the kernel of $rk$ is called the \emph{reduced} $K$ group and denoted $\tilde{K}_0(R)$.

The homomorphism $P^\text{diff}(R) \to P(R)$ obtained by forgetting the derivation is surjective \cite[Theorem 2, p. 4341]{jm2}, so  the induced map $PC^\text{diff}(R) \to PC(R)$ is an epimorphism. This map carries $FC^\text{diff}(R)$ to $0$. In fact we have the following:

\begin{proposition} \label{kerneltoK}  The map $PC^\text{diff}(R) \to PC(R)$ is an epimorphism.  The map $p: PC^\text{diff}(R)/FC^\text{diff}(R) \to PC(R)$ is an isomorphism
\end{proposition}

\begin{proof} As we have remarked, \cite[Theorem 2, p. 4341]{jm2} implies the first assertion and so $p$ is an epimorphism. To see that it is an isomorphism, we first show that it is a  homomorphism of groups and then show it has trivial kernel. To see that $PC^\text{diff}(R)/FC^\text{diff}(R)$ is a group, let $P$ be a differential finitely generated projective $R$ module. By \cite[Corollary 3 p. 4342]{jm2} there is a differential projective module $Q$ such that $P \oplus Q$ is of the form $(R^n,A)$ for some $n$ and some $A$. That is, the class of $P$ plus the class of $Q$ in $P^\text{diff}(R)$ lies in $F^\text{diff}(R)$. Thus $P^\text{diff}(R)/F^\text{diff}(R)$ is a group, hence so is its image $PC^\text{diff}(R)/FC^\text{diff}(R)$. Now suppose again that $P$ is a differential finitely generated projective module such that its class in $PC^\text{diff}(R)/FC^\text{diff}(R)$ is sent to $0$ by $p$. Thus as projective $R$ modules $P$ and $0$ have the same class in $PC(R)$, so there are $m$ and $n$ such that $P \oplus R^n$ is isomorphic to $R^m$. Thus, using the given differential structure on $P$, $P \oplus (R^n, 0)$ is a differential module, free of rank $m$ as an $R$ module, and hence of the form $(R^m, Z)$. It follows that the class of $P$ in $PC^\text{diff}(R)/FC^\text{diff}(R)$ is trivial. Thus $p$ has trivial kernel, so is injective, and hence an isomorphism.
\end{proof}

Informally, Proposition \ref{kerneltoK} can be understood as asserting that the classification of differential projective $R$ modules comes down to the classification of those which are free as $R$ modules, since when the latter are nullified the differential structures on projective modules depend only on the underlying $R$ module. More formally:

\begin{corollary} \label{trivialFCM} If $FC^\text{diff}(R)=0$ then $PC^\text{diff}(R)=PC(R)$.
\end{corollary}

\begin{corollary} \label{trivialPCMR} If $PC(R)=0$ then $FC^\text{diff}(R)=PC^\text{diff}(R)$.
\end{corollary} 

\begin{proof}  (Of Corollary \ref{trivialPCMR}.) As remarked above, it is possible for the quotient of a monoid by a proper submonoid to be trivial. We argue as in the proof of Proposition \ref{kerneltoK}. Since $PC(R)=0$, for any projective $R$ module $P$ there are $m$,$n$ such that $P \oplus R^n$ is isomorphic to $R^m$. Now suppose $P$ has a differential structure. Then so does $P \oplus (R^n, 0)$, which is then isomorphic to $(R^m, A)$ for some $A$. Since the class of $(R^m,A)$ is in $FC^\text{diff}(R)$, so is the class of $P \oplus (R^n,0)$, and this is the same as the class of $P$.
\end{proof}

Proposition \ref{kerneltoK} is also, in a way, a statement about the functor from differential modules to modules. A counterpart to that is the functor from modules over the constants to differential modules, which results in the following:

\begin{proposition} \label{units} The functor $P_0 \mapsto R \otimes_{R^D} P_0$ induces a group isomorphism $q: PC(R^D) \to U(PC^\text{diff}(R))$.
\end{proposition}

\begin{proof} The functor induces a map $P(R^D) \to P^\text{diff}(R)$. If $P_0$ is a free $R^D$ module of finite rank, then $R \otimes_{R^D} P_0$ belongs to $F_0^\text{diff}(R)$, so the map passes to a homomorphism $PC(R^D) \to PC^\text{diff}(R)$. If $P_0$ is a finitely generated projective $R^D$ module, and $Q_0$ is a finitely generated projective $R^D$ module such that $P_0 \oplus Q_0$ is free, then in $PC^\text{diff}(R)$ the classes $R \otimes_{R^D} P_0$ and $R \otimes_{R^D} Q_0$ add up to the class of $0$, and hence are inverses of each other. Thus the homomorphism $PC(R^D) \to PC^\text{diff}(R)$  has image in the units of the range.  Suppose $P$ is a differential finitely generated projective module whose class in $PC^\text{diff}(R)$ is invertible. Then there exist $m$, $n$, and a differential finitely generated projective module $Q$ such that $P \oplus Q \oplus (R^n,0)$ is isomorphic to $(R^m,0)$. In particular, $P$ is a differential direct summand of $(R^m,0)$. By \cite[Theorem 1, p. 4340]{jm2}, this means that $P$ is of the form $R \otimes_{R^D} P_0$ where $P_0$ is a finitely generated projective $R^D$ module. This shows that $q$ is group epimorphism. Suppose that the class of $P_0$ is in the kernel of $q$. Then the classes of $R \otimes_{R^D} P_0$ and $0$ are the same in $PC^\text{diff}(R)$ so there are $m$, $n$ such that $(R \otimes_{R^D} P_0) \oplus (R^n, 0)$ is differentially isomorphic to $(R^m,0)$. Now $(R \otimes_{R^D} P_0) \oplus (R^n, 0)$ is differentially isomorphic to $R \otimes_{R^D}( P_0 \oplus (R^D)^n)$ while $(R^m, 0)$ is differentially isomorphic to $R \otimes_{R^D} (R^D)^m$. Taking constants and applying \cite[Lemma 1 p. 4340]{jm2} we see that $P_0 \oplus (R^D)^n$ is $R^D$ isomorphic to  $(R^D)^m$. That is, $P_0$ and $0$ represent the same class in $PC(R^D)$. Thus $q$ has trivial kernel and hence is a group isomorphism.
\end{proof}

The functor $P_0 \mapsto R \otimes_{R^D} P_0$ also induces a group homomorphism $PC(R^D) \to PC(R)$. This need not be either injective or surjective, in general. However Proposition \ref{units} enables the identification of its kernel.

\begin{lemma} \label{inductionkernel} The kernel of $$\phi: PC(R^D) \to PC(R)$$ is identified with $U(FC^\text{diff}(R))$  via the isomorphism $q$ of Proposition \ref{units},
\end{lemma}

\begin{proof}
The inclusion $R^D \to R$ induces the group homomorphism $$\phi: PC(R^D) \to PC(R).$$ 

The kernel of $\phi$ consists of classes of finitely generated projective $R^D$ modules $P_0$ such that $P=R \otimes_{R^D} P_0$ is a stably free $R$ module: there are $m$, $n$ such that $P \oplus R^n= R \otimes_{R^D} (P_0 \oplus (R^D)^n)$ is isomorphic to $R^m$. Since $P_0$ and $P_0 \oplus (R^D)^n$ give rise to the same class in $PC(R^D)$, we can replace the former by the latter and assume that $P$ is actually $R$ free. So the class of $P$ as a differential module lies in $FC^\text{diff}(R)$. Since it comes from $PC(R^D)$ it also lies in $U(PC^\text{diff}(R))$ and hence in $I=FC^\text{diff}(R) \cap U(PC^\text{diff}(R))$. Thus $q(\text{Ker}(\phi)) \subseteq I$; since  $FC^\text{diff}(R)$ is the kernel of
the morphism $PC^\text{diff}(R) \to PC(R)$ (see Proposition \ref{kerneltoK}), we have that $q(\text{Ker}(\phi)) \cong I$. We claim that $I=U(FC^\text{diff}(R))$, which will complete the proof. That $U(FC^\text{diff}(R))$ is contained in $I$ is clear. On the other hand $I$, being the kernel of a group homomorphism, is a (sub)group in $FC^\text{diff}(R)$ and hence contained in its group of units.
\end{proof}

Using Proposition \ref{units} we have the following consequences of Corollaries \ref{trivialFCM} and \ref{trivialPCMR}:

\begin{corollary} \label{trivialFCMiso} If $FC^\text{diff}(R)=0$ then $PC(R^D) \to PC(R)$ is an isomorphism.
\end{corollary} 

\begin{proof} Suppose $FC^\text{diff}(R)=0$. By Corollary \ref{trivialFCM}, $PC^\text{diff}(R)=PC(R)$. By Proposition \ref{reducedK}, $PC(R)$ is a group, hence $PC^\text{diff}(R)=U(PC^\text{diff}(R))$, so by Proposition \ref{units} $PC(R^D)$ maps isomorphically to $PC^\text{diff}(R)$ and hence $PC(R^D) \to PC(R)$ is an isomorphism.
\end{proof}

\begin{corollary} \label{trivialPCMs} If $PC(R^D)=0$ then $PC^\text{diff}(R)$ has no invertible elements.
\end{corollary}

\begin{proof} Combine Corollary \ref{trivialPCMR} and Proposition \ref{units}.
\end{proof}

Proposition \ref{kerneltoK} implies that if $PC^\text{diff}(R)=0$ then $PC(R)=0$. Of course if $PC^\text{diff}(R)=0$ then its submonoid $FC^\text{diff}(R)=0$. In this case,  Corollary \ref{trivialFCMiso} further implies that $PC(R^D) = PC(R) =0$. We record this fact, along with the similar assertion under the stronger assumption that all differential projective $R$ modules are trivial:

\begin{corollary} \label{alltrivial} If $PC^\text{diff}(R)=0$ then  $PC(R^D) = PC(R) =0$. If all finitely generated differential projective $R$ modules are differentially trivial, then all finitely generated projective $R$ modules and all finitely generated projective $R^D$ modules are free.
\end{corollary}

\begin{proof} For the second assertion, let $P$ be a finitely generated projective $R$ module. Let $D_P$ be a differential structure on $P$. Since $P$ with $D_P$ is differentially trivial, its underlying module $P$ is free. Let $Q$ be a finitely generated projective $R^D$ module. Since $R \otimes_{R^D} Q$ is differentially trivial, it is differentially isomorphic to $(R^n,0)$ for some $n$. It follows that $Q=(R \otimes_{R^D} Q)^D=(R^n,0)^D=(R^D)^n$.
\end{proof}

In Proposition \ref{constantsfree} below we will prove a partial converse of Corollary \ref{alltrivial}

\section{Trivial Differential Projective Class Monoid}

We consider the condition that $PC^\text{diff}(R)=0$, or that all projective differential modules are stably differentially trivial. 
We will show that this condition is much more restrictive than the (non-differential) condition that all projective $R$-modules be stably free.

If $PC^\text{diff}(R)=0$, then the submonoid $FC^\text{diff}(R)$ is also trivial. A partial converse also holds:

\begin{proposition}  \label{trivialimpliestrivial}  If $FC^\text{diff}(R)=0$ and $PC(R^D)=0$ then $PC^\text{diff}(R)=0$.
\end{proposition}

\begin{proof} Suppose $FC^\text{diff}(R)=0$ and let $P$ be a finitely generated differential projective module. Then there is a differential projective $Q$ such that $P \oplus Q$ is free as an $R$ module so, as an element of $FC^\text{diff}(R)$, $P \oplus Q$  must be $0$. Thus the class of $P$ is invertible in $PC^\text{diff}(R)$. Proposition \ref{units} and the assumption that $PC(R^D)=0$ imply that the class of $P$ is $0$ in $PC^\text{diff}(R)$.
\end{proof} 

Proposition \ref{trivialimpliestrivial} says that if all differential $R$ modules which are free as $R$ modules are stably differentially trivial, and if all projective $R^D$ modules are stably free, then all differential projective $R$ modules are stably differentially trivial. This is comparable to the corresponding non--differential condition that all projective modules are stably free. The stronger condition, namely that all projective modules are free, has often been investigated; see for example \cite{lam}. The differential analogue would be the condition that all differential protectives modules are differentially trivial. As it happens, the differential and non--differential versions of this condition are related:

\begin{proposition} \label{constantsfree}  If all finitely generated projective $R^D$ modules are free, then every stably differentially trivial finitely generated differential projective $R$ module $P$ is differentially trivial.
\end{proposition}

\begin{proof} For a differential $R$ module $P$ to be stably differentially trivial means  that there are $m,n$ such that $P \oplus (R^m,0)$ is isomorphic to $(R^n,0)$. In particular, $P$ is a differential direct summand of a trivial differential module. By \cite[Theorem 1 p. 4340]{jm2} this implies that $P$ is of the form $R \otimes_{R^D} P_0$ for some finitely generated projective $R^D$ module $P_0$. Since by assumption $P_0$ is free, say for rank $k$ (where of course $k=n-m$), $P$ is differentially isomorphic to $(R^k, 0)$ and hence differentially trivial. 
\end{proof}

If $\eta: R \to S$ is a differential ring homomorphism, then there is an induced monoid homomorphism $h: PC^\text{diff}(R) \to PC^\text{diff}(S)$ which carries the submonoid $FC^\text{diff}(R)$ of $PC^\text{diff}(R)$ to the submonoid 
$FC^\text{diff}(S)$ of $PC^\text{diff}(S)$. Even if $\eta$ is surjective, $h$ need not be surjective, but its restriction  to $FC^\text{diff}(R)$ does surject onto $FC^\text{diff}(S)$. If it happens that $PC(S^D)=0$, then we can use this to detect whether $PC^\text{diff}(R) \neq 0$.

\begin{corollary} \label{nontivialimpliesnontrivial} Let $\eta: R \to S$ be a surjective differential homomorphism of differential rings. Assume $PC(S^D)=0$. If $PC^\text{diff}(R) = 0$ then $PC^\text{diff}(S) = 0$.
\end{corollary} 

\begin{proof}  If $PC^\text{diff}(R)=0$, then  $FC^\text{diff}(R)=0$. Since $\eta$ is surjective, this implies that $FC^\text{diff}(S)=0$. Then by Proposition \ref{trivialimpliestrivial} again, $PC^\text{diff}(S)=0$.
\end{proof}

If $S$ is a simple differential ring, then $S^D$ is a field \cite[Lemma 1.3.6, p. 11]{w}, so all projective $S$ modules are free, and in particular $PC(S^D)=0$. So if $R \to S$ is a surjection with $S$ simple, and if $PC^\text{diff}(R)=0$, then $PC^\text{diff}(S)=0$.

\begin{corollary} \label{simplequotient} Let $R$ be a differential ring with $PC^\text{diff}(R)=0$. Then for every maximal differential ideal $I$ of $R$, $PC^\text{diff}(R/I)=0$.
\end{corollary}

Note that in general $R/I$, for $I$ a maximal differential ideal, may have differential projective modules which are not obtained from differential projective $R$ modules by reduction modulo $I$. Indeed every differential $R$ module, projective or not, reduces to a differential projective $R/I$ module since every finitely generated  differential module over the simple differential ring $R/I$ is a projective $R/I$ module \cite[Theorem 2.2.1. p. 455]{a}. 

We consider as examples a class of differential rings whose differential projective class monoids are \emph{never} trivial.

We will be using some facts about Picard--Vessiot extensions, which we now briefly recall (\cite[Chapter 1]{svp}). Let $F$ be a characteristic zero differential field with algebraically closed field of constants $C$. Let $P=(F^n ,A)$ be a finitely generated differential projective $F$ module. A  differential field extension $E \supseteq F$ is called a \emph{Picard--Vessiot extension} of $F$ for $P$ provided: (1) $E^D=C$; (2) $E \otimes_F P$ is differentially trivial; and (3) No proper differential subfield of $E$ containing $F$ satisfies (1) and (2). Picard--Vessiot extensions exist and are unique up to $F$-isomorphism. Now suppose $P_1$ and $P_2$ are finitely generated projective differential $F$ modules which represent the same class in $PC^\text{diff}(F)$. That is, there are $m,n$ such that $P_1 \oplus (F^m, 0)$ and $P_2 \oplus (F^n, 0)$ are differentially isomorphic. Let $E_i$ be a Picard--Vessiot extension for $P_i$. Since $E_1 \otimes_F P_1$ has a basis of constants, so does 
$E_1 \otimes_F (P_1 \oplus (F^m,0))$. Because of the isomorphism, $E_1 \otimes_F (P_2 \oplus (F^n,0))$ also has a basis of constants. In other words, $E_1 \otimes_F (P_2 \oplus (F^n,0))$ is differentially trivial. Thus $E_1 \otimes _F P_2$ is stably differentially trivial, and then by Proposition \ref{constantsfree} $E_1 \otimes_F P_2$ is differentially trivial, i.e. has a basis of constants. Since also $E_1^D=C$, $E_1$ contains a Picard--Vessiot extension for $P_2$. Similarly, $E_2$ contains a Picard--Vessiot extension of $P_1$. We conclude that $E_1$ and $E_2$ are isomorphic differential $F$ algebras. Suppose we have chosen all Picard--Vessiot extensions of $F$ inside the same Picard--Vessiot closure of $F$. That is, inside a differential field extension of $F$ with the same constants as $F$ which is a union of Picard--Vessiot extensions of $F$ and which contains an isomorphic copy of every Picard--Vessiot extension of $F$. (See \cite{m3} for the proof that Picard--Vessiot closures exist.) Then we have that $E_1=E_2$. Conversely, if the Picard--Vessiot extensions for $P_1$ and $P_2$ are not equal, then $P_1$ and $P_2$ have different classes in $PC^\text{diff}(F)$. This same reasoning applies to sets of differential $F$ modules with more than two members, indeed to such sets of any cardinality:

\begin{proposition} \label{bigmonoid} Let $F$ be a characteristic zero differential field with algebraically closed field of constants $C$. Let $\{ P_i \vert i \in I \}$ be a set of finitely generated differential projective $F$ modules and for each $i \in I$ let $E_i$ be a Picard--Vessiot extension of $F$ for $P_i$, all inside a chosen Picard--Vessiot closure of $F$. Suppose all the $E_i$'s are distinct. Then the classes of the $P_i$'s are distinct elements of $PCM^\text{diff}(F)$.
\end{proposition} 

Next we look at differential algebras which are finitely generated as $F$ algebras over differential fields $F$ that admit Picard--Vessiot extensions of arbitrarily large transcendence degree (this is a relatively mild condition satisfied for example by $F= \mathbb C$ with $0$ derivation and $F=\mathbb C(x)$ with derivation $d/dx$) :

\begin{proposition} \label{affine} Let $F$ be a characteristic zero differential field with algebraically closed field of constants $C$. Suppose that $F$ admits Picard--Vessiot extensions of arbitrarily large transcendence degree over $F$. Let $R$ be a differential $F$ algebra which is finitely generated as an $F$ algebra. Then $PC^\text{diff}(R) \neq 0$.
\end{proposition}

\begin{proof} By Corollary \ref{simplequotient} it is enough to show that for some maximal differential ideal $I$ of $R$ $PC^\text{diff}(R/I) \neq 0$. Thus we can replace $R$ by $R/I$ and assume that $R$ is differentially simple. It follows that the quotient field $E$ of $R$ is an extension of $F$ whose field of constants is also $C$ \cite[Corollary 1.18, p. 11]{m2}. Note that $E$ has finite transcendence degree over $F$, so there are finitely generated differential modules over $F$ 
which do not admit Picard--Vessiot extensions in $E$. In particular, this means there is some matrix $A$ over $F$ such that $(F^n, A)$ does not admit a basis of constants over $E$, and hence does not have one over $R$. Thus $(F^n,A)$ is not a trivial differential $R$ module. Since $R^D$ is the field $C$, all projective $R^D$ modules are free. Proposition \ref{constantsfree} then implies that $(F^n,A)$ is not even stably differentially trivial, and hence its class in $PCM^\text{diff}(R)$ is non--trivial.
\end{proof}

Proposition \ref{affine} applies to differential affine $\mathbb C$ algebras, none of which, including polynomial rings or Laurent polynomial rings, can have trivial differential projective class monoids. 

\section{Cancellation}

Throughout this section we assume that $R$ is connected, so that every projective $R$ module is of constant rank; in fact, we only need this for finitely generated differential projective $R$ modules. The assumption guarantees that the number of
non-zero summands in a direct sum decomposition of a module of rank $r$ is at most $r$.

\begin{definition} \label{D:Rcancellation} $R$ is said to have \emph{free cancellation} if, for finitely generated differential projective $R$ modules $P$ and $Q$, if $P \oplus (R,0)$ is differentially isomorphic to $Q \oplus(R,0)$ then $P$ is differentially isomorphic to $Q$.
\end{definition}

From the definition and induction, using the identification of $(R^n, 0)$ with $(R,0)^{(n)}$, we see that  for a ring with free cancellation cancellation of $(R,0)$  implies the cancellation of $(R^n,0)$. 

\begin{proposition} \label{P:cancelRn} Let $R$ have free cancellation and let $P$ and $Q$ be finitely generated differential projective $R$ modules with $P \oplus (R^n,0)$ differentially isomorphic to $Q \oplus (R^n,0)$. Then $P$ is differentially isomorphic to $Q$.
\end{proposition}

It is a corollary of Proposition \ref{P:cancelRn} that differentially trivial modules of different ranks may be canceled:

\begin{corollary}  \label{C:cancelRn} Let $R$ have free cancellation and let $P$ and $Q$ be finitely generated differential projective $R$ modules with $P \oplus (R^n,0)$ differentially isomorphic to $Q \oplus (R^m,0)$ where $n \geq m$. Then $P$ is differentially isomorphic to $Q \oplus (R^{n-m},0)$.
\end{corollary}

If $n=m$ in Corollary \ref{C:cancelRn} and below then by $(R^0,0$) is meant $0$.

\begin{definition} \label{D:core} Let $P$ be a finitely generated differential projective $R$ module. If $P$ has no differential direct summand isomorphic to $(R,0)$ then $P$ is said to be  \emph{trivial--free}. 
A differential projective submodule $C$ of $P$ is called a \emph{core} provided
\begin{itemize}
\item[1.] $C$ is trivial--free; and\\
\item[2.] $C \oplus (R^n,0)$ is differentially isomorphic to $P$ for some $n$ (possibly $n=0$).
\end{itemize}
\end{definition}

By condition $2.$, $C$ and $P$ represent the same class in $PC^\text{diff}(R)$.

The limit on the number of summands of a finitely generated projective module implies that every finitely generated differential projective module has a core. For example, the zero module is a core of $(R^n,0)$.

In general, a module may have non-isomorphic cores. For example, in the zero derivation case, a non-free stably free module will lead to an example with non-isomorphic cores: see Example \ref{E:nocancellation} below. In the free cancellation case, this doesn't happen:

\begin{proposition} \label{P:uniquecore} Assume $R$ has free cancellation and let $P$ be a finitely generated differential projective module $R$ module. Let $C$ and $D$ be cores of $P$. Then $C$ and $D$ are differentially isomorphic.
\end{proposition}

\begin{proof} For some $n, m$ we have both $C \oplus (R^n,0)$ and $D \oplus(R^m,0)$ isomorphic to $P$ and hence to each other. Without loss of generality we may assume $n \geq m$. By Corollary \ref{C:cancelRn} $C$ is differentially isomorphic to $D \oplus (R^{n-m}, 0)$. Since $C$ has no summand isomorphic to $(R,0)$ we conclude that $n=m$ and hence that by Proposition \ref{P:cancelRn} $C$ and $D$ are differentially isomorphic.
\end{proof}

Using Proposition \ref{P:uniquecore} we introduce, in the case where $R$ has free cancellation, the following notation:

\begin{notation} \label{N:H}    Let $R$ have free cancellation and let $P$ be a finitely generated differential projective $R$ modules. Then $H(P)$ denotes  a core of $P$; $H(P)$ is unique up to isomorphism. \end{notation} 

If $P$ is trivial--free then $H(P)=P$.

Note that $H((R^n,0))=\{0\}$.

Suppose that $R$ has free cancellation and let $P$ be a finitely generated differential projective module. Then $P$ and $H(P)$ belong to the same class in $PC^\text{diff}(R)$; in fact $H(P)$ is a module of minimal rank in the class of $P$. This sets up a bijection between the isomorphism classes of trivial--free finitely generated projective differential modules and $PC^\text{diff}(R)$. This bijection means a monoid structure can be imposed on the set of isomorphism classes of trivial--free finitely generated projective differential modules, namely if $H(P)=P$ and $H(Q)=Q$ then $[P] + [Q]= [H(P\oplus Q)]$. This coincides with the operation on isomorphism classes induced from direct sums provided the direct sum of trivial--free modules is trivial--free. 

We turn now to a class of differential rings which has free cancellation and also this property for direct sums.

\begin{proposition} \label{T:fieldconstantscancel} Suppose $R^D$ is a field. Then $R$ has free cancellation.
\end{proposition}

\begin{proof} Let $C$ denote the field $R^D$. Let $M$ be a finitely generated differential projective $R$ module, and suppose $P=M \oplus (R,0)$ can be written as an internal direct sum of differential submodules $N_1$ and $R_1$ where $R_1$ is differentially isomorphic to $(R,0)$. To prove the proposition, we must show that $M$ is differentially isomorphic to $N_1$. 

We write $P=M \oplus (R,0)$ as $\{(m,r) | m\in P, r \in R\}$, and regard $N_1$ and $R_1$ as differential submodules of the latter with $N_1 + R_1 =P$ and $N_1 \cap R_1 =0$. By assumption, $(R,0)$ is differentially isomorphic to $R_1$; let $x=(y,z)$ be the $R$ generator of $R_1$ corresponding to $1$. Since $D(1)=0$, $D(x)=(D(y), D(z))=(0,0)$. Thus $D(y)=0$ and $D(z)=0$ so $z \in C$. 

We consider first the case that $z \neq 0$.

Then $z$ is a unit of $R$, and $R_1=Rx=Rz^{-1}x$ where $z^{-1}x=(w,1)$. Note that $D(w)=D(z^{-1}y)=0$.  Define $T:  P \to P$ by $(m,r) \mapsto (m+rw,r)$. $T$ is a differential endomorphism, and in fact an automorphism with inverse $S$ given by $(m,r) \mapsto (m-rw,r)$. Then $T((0,1))=(w,1)$ which implies that $T(0 \oplus (R,0))=R_1$ and $S(R_1) = 0 \oplus (R,0)$. Also $T(M \oplus 0)=M \oplus 0$; in fact $T$ is the identity on $M \oplus 0$. Let $N=S(N_1)$. Then $N + (0 \oplus (R,0))=S(N_1)+S(R_1)=S(N_1+R_1)=S(P)=P$. Let $(m,r)$ be any element of $P$. Since $P=N+(0\oplus (R,0))$, $(m,r)=n+s(0,1)$ for some $n \in N$ and $s \in R$. Thus $n=(m,r-s)$. In particular, every element of $M$ appears as the first entry of an element of $N$. Also $N \cap (0 \oplus (R,0))=S(N_1) \cap S(R_1)=S(N_1 \cap R_1)=0$. The latter implies that if $(0,r) \in N$ then $r=0$. So if $(m,s) \in N$ and $(m,t) \in N$ then $(m,s)-(m,t)=(0,s-t) \in N$ so $s=t$. That is, the second entry of an element of $N$ is determined by the first entry. Thus we have a function $f:M \to R$ such that $N=\{(m, f(m))|m \in M\}$. Since $N$ is a differential submodule of $P$, $f$ is a differential homomorphism. It follows that $M \to N$ by $m \mapsto (m,f(m))$ is a differential homorphism, in fact a differential isomorphism whose inverse $N \to M$ is given by $(m,f(m)) \mapsto m$. This proves the result in the case $z \neq 0$.

Now consider the case that $z=0$. 

Then $R_1=R(y,0)$ where $D(y)=0$. This implies $R_1 \subset M \oplus 0$.Thus $$(N_1 \cap (M \oplus 0))+R_1 \subseteq M\oplus 0.$$ Let $m \in M$. Since $N_1+R_1=P$, there are $n_1 \in N_1$ and $r_1=r(y,0) \in R_1$ such that $(m,0)=n_1+r_1$. That is $n_1=(m-ry,0) \in M\oplus 0$. Thus $n_1 \in N_1 \cap (M \oplus 0)$. Thus $M \oplus 0=(N_1 \cap (M \oplus 0))+R_1$. Let $N_2=N_1 \cap (M \oplus 0)$. Then $M \oplus 0$ is the internal direct sum of $N_2$ and $R_1$. ($N_2 \cap R_1 = \{0\}$ since $N_1 \cap R_1=\{0\}$.) Reverting to direct sum notation, we have that $M$ is the direct sum $N_2 \oplus R_1$ which means that $P$ is the direct sum $N_2 \oplus R_1 \oplus (R,0)$, while $P$ is also the direct sum to $N_1 \oplus R_1$. Thus on the one hand we have 
$P/R_1$ differentially isomorphic to $N_2 \oplus (R,0)$ and on the other that $P/R_1$ is differentially isomorphic to $N_1$. We conclude that $N_1$ is differentially isomorphic to $N_2 \oplus (R,0)$. Since the latter is differentially isomorphic to $N_2 \oplus R_1$, and this in turn is differentially isomorphic to 
$M$, we have that $M$ and $N_1$ are differentially isomorphic, completing the proof in the case $z=0$ and hence of  the proposition.
\end{proof}

Next, we see that the same condition on $R$ -- that $R^D$ is a field -- which Proposition \ref{T:fieldconstantscancel} shows implies free cancellation also implies that the direct sum of trivial--free modules is trivial--free.

\begin{proposition} \label{T:fieldconstantstrivialfree} Suppose $R^D$ is a field. Then if $P_1$ and $P_2$ are finitely generated differential projective modules such that $P_1 \oplus P_2$ has a differential direct summand differentially isomorphic to $(R,0)$ then either $P_1$ or $P_2$ has such a summand.
\end{proposition}

\begin{proof} Suppose there is a direct summand as in the proposition Let $g:(R,0) \to P_1 \oplus P_2$ and $f:P_1 \oplus P_2 \to (R,0)$ be differential homomorphisms with $fg$ the identity. Let $f_i:P_i \to (R,0)$ be $f$ preceded by inclusion on summand $i$ and let $g_i:(R,0) \to P_i$ be $g$ followed by projection on summand $i$. For each $i$, $f_ig_i$ is a differential endomorphism of $(R,0)$ and hence given by multiplication by a constant $a_i$. For $r \in (R,0)$, $r=fg(r)=f_1(g_1(r))+f_2(g_2(r))$. Thus at least one of the $f_ig_i$ is non-zero, say for $i=1$. Thus the constant $a_1$ is non-zero, and since $R^D$ is a field $a_1$ is a unit. It follows that $(a_1^{-1}f_1)g_1$ is the identity, showing that $g_1((R,0))$ is a differential direct summand of $P_1$ isomorphic to $(R,0)$.
\end{proof}

\begin{corollary} \label{sumoftrivialfree} Suppose $R^D$ is a field. Let $P$ and $Q$ be finitely generated differential projective $R$ modules. Then $H(P) \oplus H(Q)$ is differentially isomorphic to $H(P \oplus Q)$.
\end{corollary}

\begin{proof} By Proposition \ref{T:fieldconstantscancel}, $R$ has cancellation, so both $P$ and $Q$ have cores unique up to differential isomorphism, so the symbols $H(P)$ and $H(Q)$ make sense. For some $m,n$ we have $P$ isomorphic to $H(P) \oplus (R^m,0)$ and $Q$ isomorphic to $H(Q) \oplus (R^n,0)$. Thus $P \oplus Q$ is differentially isomorphic to $H(P) \oplus H(Q) \oplus (R^{m+n},0)$. By Proposition \ref{T:fieldconstantstrivialfree}, or rather its contrapositive, $H(P) \oplus H(Q)$ is trivial--free, and the previous differential isomorphism shows that it is a core of $P \oplus Q$. 
\end{proof}

Corollary \ref{sumoftrivialfree} implies that, in the case that $R^D$ is a field,  (1) the isomorphism classes of trivial--free finitely generated projective differential modules form a monoid with the operation induced from direct sum; and (2) that this monoid is isomorphic to $PC^\text{diff}(R)$:

\begin{theorem} \label{T:trivialfreeareall} Suppose $R^D$ is a field. Then
\begin{itemize}
\item[1.] The differential isomorphism classes of trivial--free finitely generated projective differential $R$ modules form a monoid $M$ with the operation induced from direct sum. \\
\item[2.] $P \mapsto H(P)$ induces an isomorphism $PC^\text{diff}(R) \to M$.
\end{itemize}
\end{theorem}

\begin{proof} For assertion $1.$, suppose $P$ and $Q$ are trivial--free finitely generated  projective differential $R$ modules. Then $H(P)$ is differentially isomorphic to $ P$ and $H(Q)$ is differentially isomorphic to $Q$. Since $H(P \oplus Q)$ is  differentially isomorphic to $H(P) \oplus H(Q)$ by Corollary \ref{sumoftrivialfree}, $H(P \oplus Q)$ is differentially isomorphic to $P \oplus Q$, which means $P \oplus Q$ is trivial free. Thus the monoid $M$ is well-defined. For assertion $2.$, we begin by noting that
by Proposition \ref{P:uniquecore} and Proposition \ref{T:fieldconstantscancel} $P \mapsto H(P)$ defines an endomorphism of
 the monoid $P^\text{diff}(R)$ of isomorphism classes of finitely generated  projective differential $R$ modules with the operation induced from direct sum as in Definition \ref{themonoids}. If $P$ and $Q$ are finitely generated  projective differential $R$ modules and there are $m,n$ with $P \oplus (R^n,0)$ differentially isomorphic to $Q \oplus (R^m,0)$  then by Corollary \ref{sumoftrivialfree} $H(P \oplus (R^n,0))$ is isomorphic to $H(P) \oplus H((R^n,0))$ and since the second summand is $\{0\}$ we conclude that $H(P\oplus (R^n,0))$ is differentially isomorphic to $H(P)$. Similarly, $H(Q \oplus (R^m,0))$ is differentially isomorphic to $H(Q)$. Hence $H(P)$ is differentially isomorphic to $H(Q)$. Thus the monoid homomorphism $P^\text{diff}(R) \to M$ induced by $P \mapsto H(P)$ factors through $PC^\text{diff}(R)$ and defines a monoid homomorphism $h:PC^\text{diff}(R) \to M$. To see that $h$ is onto, let $P$ be a trivial--free finitely generated differential projective $R$ module. Then since $H(P)$ is differentially isomorphic to $P$, $h$ applied to the class of $P$ in $PC^\text{diff}$ is the class of $P$ in $M$. So $h$ is onto. Suppose $P$ and $Q$ are sent to the same element by $h$. That means $P$ and $Q$ have differentially  isomorphic cores. Let $C$, $D$ be cores of $P$, $Q$ respectively. Then by Definition \ref{D:core} $2.$  there are $m,n$ such that $P$ is differentially isomorphic to $C \oplus (R^n,0)$ and $Q$ is differentially isomorphic to $D \oplus (R^m,0)$. Since $C$ and $C \oplus (R^n,0)$ (respectively $D$ and $D \oplus (R^m,0)$) represent the same element of $PC^\text{diff}(R)$, and so do $C$ and $D$, we have that $P$ and $Q$ represent the same element of $PC^\text{diff}(R)$. This shows that $h$ is injective and hence an isomorphism.
 \end{proof}
 
 It is a consequence of Theorem \ref{T:trivialfreeareall} that if $P$ and $Q$ are non-isomorphic trivial--free finitely generated projective differential $R$ modules, for instance if they have different ranks, then 
 since they represent different classes in $M$ they represent different classes in $PC^\text{diff}(R)$.
 
 We can think of $P \mapsto H(P)$ as inducing an endomorphism of $P^\text{diff}(R)$ which is the identity on $M$ and whose image is $M$. Theorem \ref{T:trivialfreeareall} says that $M$ is isomorphic to $PC^\text{diff}(R)$. Restated:

\begin{corollary} \label{C:split} Suppose $R^D$ is a field. Then the monoid epimorphism $$P^\text{diff}(R) \to PC^\text{diff}(R)$$ has a right inverse.
\end{corollary}

The monoid epimorphism $P^\text{diff}(R) \to PC^\text{diff}(R)$ is injective on $M$ and hence on submonoids of $M$. As an example of some calculations using this observation, we consider the case $R=\mathbb C[x]$ with $D=d/dx$. We have $R^D= \mathbb C$ is a field, so Theorem \ref{T:trivialfreeareall} applies. We consider modules of the form $(R,f)$ for $f \in R$, where we have identified $1 \times 1$ matrices with their only entry. A differential homomorphism $(R,f) \to (R,g)$ is given by multiplication by $a \in R$ satisfying $a^\prime= fa-ag$. Thus $a^\prime=a(f-g)$ which, considering polynomial degrees,  is only possible if $a=0$ or, if $a \neq 0$, when both $a \in \mathbb C$ and $f=g$. Only the second possibility holds if $(R,f) \to (R,g)$ is an isomorphism. In particular, if $f \neq 0$ then $(R,f)$ is not isomorphic to $(R,0)$, implying that $(R,f)$ is trivial free. Thus if $f_1, \dots, f_n \in R$ are all non-zero, by Corollary \ref{sumoftrivialfree}  $P=(R,f_1) \oplus \dots \oplus (R,f_n)$ is trivial--free. Similarly, if  $g_1, \dots, g_m$ are non-zero elements of $R$ then $Q=(R,g_1) \oplus \dots \oplus (R,g_m)$ is trivial--free. Thus $P$ and $Q$ will represent different classes in $PC^\text{diff}(R)$ unless $P$ and $Q$ are isomophic. This requires that $m=n$, and then we can represent a differential isomorphism $T:P \to Q$ by an $n \times n$ matrix whose entries are  differential homomorphisms $(R,f_i) \to (R,g_j)$. As we have seen, these are all given by multiplication by elements of $\mathbb C$, where the elements are zero unless $f_i=g_j$. Thus if no $f_i$ equals any $g_j$ then $P$ and $Q$ are not differentially isomorphic. We can use these observations to produce uncountably many distinct elements of the monoid of isomorphism classes of trivial--free finitely generated differential projective $R$ modules and hence of $PC^\text{diff}(R)$. Of course in this case both $PC(R^D)=0$ and $PC(R)=0$ as well.

\section{Zero Derivation Rings}

Any commutative ring can be made the underlying ring of a differential ring by giving it the zero derivation. A derivation on a module over a zero derivation ring is just a linear endomorphism, so any module over a zero derivation ring can be made into a differential module by choosing such an endomorphism. Thus examples of projective modules, such as non--free projective modules, or non--free stably free projective modules, can be used to give examples of differential projective modules whose underlying projective modules have these properties. In particular, the endomorphism chosen could be the zero map (multiplication by $0$), the identity map (multiplication by $1$) or multiplication by any other scalar, as we do in the following result.

\begin{proposition} \label{P:zeroderivative} Let $A, B$, and $C$ be modules over the commutative ring $R$. Give $R$ the zero derivation and let $A$ and $B$ be given differential structures by using their identity endomorphisms. Let $C$ be given differential structure  by the zero endomorphism. Suppose $A \oplus C$ is differentially isomorphic to $B \oplus C$. Then $A$ is differentially isomorphic to $B$.
\end{proposition}

\begin{proof} Let $f: A \oplus C \to B \oplus C$ be a differential isomorphism. Since $(A \oplus C)^D$ is a differential submodule of $A \oplus C$, $(B \oplus C)^D$ is a differential submodule of $B \oplus C$,  and $f((A \oplus C)^D) = (B \oplus C)^D$, $f$ induces a differential isomorphism $(A\oplus C)/(A \oplus C)^D \to (B \oplus C)/(B \oplus C)^D$. Since $(A \oplus C)^D= \{0\} \oplus C$ and $(B \oplus C)^D= \{0\} \oplus C$
this induced isomorphism is a differential isomorphism between $A$ and $B$.
\end{proof}

In Proposition \ref{P:zeroderivative} if $C$ is the free module $R^n$ then $C$ with the zero derivative is $(R^n,0)$. Thus if $A$ and $B$ have the identity derivation, and if $A \oplus (R^n,0)$ is differentially isomorphic to $B \oplus (R^n,0)$, then $A$ is differentially isomorphic to $B$. 

\begin{example} \label{E:nocancellation} Let $R$ be a connected commutative ring and let $A$ be a non-free stably free $R$ module; for example, \cite[Example 1, p. 269]{sw}. That is, there are $m,n$ such that there is an $R$ isomorphism $f: A \oplus R^m \to R^n$. If follows that $A$ is projective of rank $p=n-m$. Moreover $f$ can be regarded as an isomorphism from $A \oplus R^m$ to $R^p \oplus R^m$. However $A$ is not isomorphic to $R^p$. Now give $R$ the zero derivation structure and $A$, $R^p$, and $R^m$ the differential structure coming from the identity. Then $f$ is a differential homomorphism. If we let $B$ denote $R^p$ with the identity derivation and $P$ denote $R^m$ with the identity derivation then $A \oplus P$ is differentially isomorphic to $B \oplus P$, but, since $A$ and $B$ are not isomorphic, Proposition \ref{P:zeroderivative} shows that for no $q$ is $A \oplus (R^q,0)$ isomorphic to $B \oplus (R^q,0)$.
\end{example}

In the notation of Example \ref{E:nocancellation}, we see that the differential projective $R$ modules $A$ and $B$ give rise to distinct elements of $PC^\text{diff}(R)$ but the same element of $K_0^\text{diff}(R)$.

We have recalled the fact \cite[Theorem 2, p. 4341]{jm2}  that for any differential ring $R$ every finitely generated projective $R$ module $M$ can be endowed with  a differential structure. This makes the forgetful functor $\mathcal U$ from the additive category of finitely generated differential projective $R$ modules to the additive category of finitely generated projective $R$ modules surjective on objects. In the case that $R$ has the zero derivation, as we noted above, we can make a projective module $M$ a differential module using the zero endomorphism of $M$. And if $M$ and $N$ are both finitely generated projective $R$ modules and $f:M \to N$ an $R$ module homomorphism, then $f$ is a differential homomorphism when both $M$ and $N$ are given differential structure using their zero endomorphisms. That is, using the zero endomorphism differential structures produces an additive functor $\mathcal F$ from the category of finitely generated projective $R$ modules to the category of 
finitely generated differential projective $R$ modules with $\mathcal U\mathcal F$ the identity. Note that $\mathcal F(R)=(R,0)$. Could such a functor exist in general for any differential commutative ring $R$? 

More precisely: Let $R$ be a differential ring and let $\mathcal U$ denote the forgetful functor from the additive category of finitely generated differential projective $R$ modules to the additive category of finitely generated projective $R$ modules. We ask: can there be an additive functor $\mathcal F$ from the additive category of finitely generated projective $R$ modules to the additive category of finitely generated differential projective $R$ modules with
\begin{itemize}
\item[1.] $\mathcal U\mathcal F$ naturally equivalent to the identity; and \\
\item[2.] $\mathcal F(R)=(R,0)$.\\
\end{itemize} 

The additive functors $\mathcal U$ and $\mathcal F$ will induce monoid homomorphisms $$u: P^\text{diff}(R) \to P(R)$$  $$f: P(R) \to P^\text{diff}(R)$$ which, by condition $1.$, are such that $uf$ is the identity. Since $\mathcal U((R^n,0))=R^n$, $u$ induces a monoid homomorphism 
$\overline{u}:PC^\text{diff}(R) \to PC(R)$. An additive functor $\mathcal F$ satisfying condition $2.$ will have $\mathcal F(R^n)=(R,0)^{(n)}$ which is differentially isomorphic to $(R^n,0)$.  Thus $f$ carries $F(R)$ to $F_0^\text{diff}(R)$ and hence induces a monoid homomorpism $\overline{f}: PC(R) \to PC^\text{diff}(R)$ with $\overline{u} \overline{f}$ equaling the identity.  Since (by Proposition \ref{reducedK}) $PC(R)$ is a group, so is its isomorphic image  $\overline{f}(PC(R))$. Now suppose that $R$ is such that $PC(R^D)=0$ and $PC(R) \neq 0$. Then by Corollary \ref{trivialPCMs} $PC^\text{diff}(R)$ has no invertible elements, so its subgroup $\overline{f}(PC(R))$ is trivial. This contradiction means that for such a ring $R$ there is no such functor $\mathcal F$. For a concrete example of such an $R$ we can take the Picard--Vessiot ring $R$ of a Picard--Vessiot extension of a differential field $F$ with differential Galois group $G$ such that the coordinate ring $F[G]$ has non stably free projective modules. Then $R^D=F$ so $PC(R^D)=0$ while $PC(R)=PC(F[G]) \neq 0$.

This paper has been published in \emph{Communications in Algebra} 

 https://doi.org/10.1080/00927872.2022.2045493.


\begin{thebibliography}{9999}


\bibitem{a} Andr\'e, Y.  \emph{Solution algebras of differential equations and quasi--homogeneous varieties}, Ann. Sci. \'Ecole Norm. Sup. \textbf{47}, 449-467


\bibitem{b2} Bass, H. \emph{Algebraic K Theory}, W. A. Benjamin, New York, 1968



\bibitem{jm}  Juan, L., and Magid, A. \emph{Differential central simple algebras and Picard--Vessiot representations}, Proc. Amer. Math. Soc. \textbf{136} (2008), 1911-1918

\bibitem{jm2} Juan, L., and Magid, A. \emph{Differential projective modules over differential rings}, Communications in Algebra \textbf{47} (2019), 4336-4346

\bibitem{lam} Lam, T. Y. \emph{Serre's Problem}, Springer, Berlin, 2006

\bibitem{m} Magid, A. \emph{Differential Brauer monoids}, ArXiv 123456

\bibitem{m2} Magid, A. \emph{Lectures on Differential Galois Theory}, University Lecture Series 7, American Mathematical Society, Providence 1994 (second printing with corrections, 1997)

\bibitem{m3} Magid, A. \emph{The Picard-Vessiot closure in differential Galois theory. Differential Galois theory} (Bedlewo, 2001), 151-164, Banach Center Publ., 58, Polish Acad. Sci. Inst. Math., Warsaw, 2002.

\bibitem{svp} Singer, M., and van der Put, M. \emph{Galois Theory of Linear Differential Equations}, G. der math. Wiss. 328, Springer, Berlin, 2003

\bibitem{sw} Swan, R. S., \emph{Vector Bundles and Projective Modules}, Trans. Amer. Math Soc. \textbf{105} (1962), 264-277

\bibitem{w} Wibmer, M. \emph{Differential Galois theory, an introduction to the Galois theory of linear differential equations} , retrieved from https://sites.google.com/view/wibmer


\end{thebibliography}
\end{document}